\documentclass[12pt,A4paper]{amsart}
\usepackage{amsmath,amsfonts,amsthm,amssymb,color,hyperref}

\usepackage[T1]{fontenc}

\usepackage{graphicx}

\usepackage{subfigure}
\usepackage[foot]{amsaddr}

\makeatletter
     \def\section{\@startsection{section}{1}%
     \z@{.7\linespacing\@plus\linespacing}{.5\linespacing}%
     {\bfseries
     \centering
     }}
     \def\@secnumfont{\bfseries}
     \makeatother

\usepackage{graphicx}

\newcommand{\R}{\mathbb R}

\newcommand{\E}{\mathbb E}

\newcommand{\1}{{\bf 1}}

\newtheorem{theorem}{Theorem}[section]
\newtheorem{assumption}{Assumption}[section]
\newtheorem{lemma}[theorem]{Lemma}
\newtheorem{proposition}[theorem]{Proposition}
\newtheorem{corollary}[theorem]{Corollary}
\theoremstyle{definition}
\newtheorem{definition}[theorem]{Definition}
\theoremstyle{remark}
\newtheorem{remark}{Remark}
\numberwithin{equation}{section}
\begin{document}
\title[Oscillating Processes]{Oscillating Gaussian Processes}
\author[P.Ilmonen]{Pauliina Ilmonen$^1$}
 \address{1: Aalto University School of Science, Department of Mathematics and Systems Analysis, Finland. }
  \email{pauliina.ilmonen@aalto.fi}
\author[S.Torres]{Soledad Torres$^2$}
\address{2: Facultad de Ingenier\'ia, CIMFAV  Universidad de Valpara\'iso, Casilla 123-V, 4059 Valparaiso, Chile. }
 \email{soledad.torres@uv.cl}
\author[L.Viitasaari]{Lauri Viitasaari$^3$}
\address{3: Aalto University School of Business, Department of Information and Service Management, Finland (\textbf{Corresponding author})}
\email{lauri.viitasaari@iki.fi}

\begin{abstract}
In this article we introduce and study oscillating Gaussian processes defined by $X_t = \alpha_+ Y_t \1_{Y_t >0} + \alpha_- Y_t\1_{Y_t<0}$, where $\alpha_+,\alpha_->0$ are free parameters and $Y$ is either stationary or self-similar Gaussian process. We study the basic properties of $X$ and we consider estimation of the model parameters. In particular, we show that the moment estimators converge in $L^p$ and are, when suitably normalised, asymptotically normal.
\end{abstract}

\maketitle



\allowdisplaybreaks

\section{Introduction}
During the past two decades interest in the study of the existence and uniqueness of stochastic differential equations driven by a fractional Brownian motion has been very intense and there have been many advances in their theory and applications.  In particular, strong solutions of  the following stochastic differential equation (SDE in short)
\begin{equation}\label{sde}
X_t = X_0 + \int_0^t b(s,X_s) ds + \int_0^t ¡\sigma(s,X_s) dB^H_s,
\end{equation}
under usual conditions on the coefficients,  such as Lipschitz and linear growth,  were developed by Nualart and R$\check{a}\text{\c{s}}$canu \cite{NualartRascanu2002}, and have been considered by many authors, see  \cite{Mishura} and the references therein.

Nevertheless, the case of SDE with  discontinuous coefficients has been less explored.  Most of the cases of stochastic differential equations driven by a fractional Brownian motion and with discontinuous coefficients which have been studied are those corresponding to discontinuous drift coefficient (for  $H>1/2$).  Regarding that, in  \cite{MN}, the authors studied a drift that is H\"older continuous except on a finite numbers of points. Another class of discontinuity in SDE driven by a fractional Brownian motion is related to adding a Poisson process to the equation. In \cite{BM}, extending  the results given in \cite{MN}, the authors proved the existence of the strong solution of this kind of SDE driven by a fractional Brownian motion and a Poisson point process. To the best of our knowledge, in the fractional Brownian motion framework, there is only a preliminary  work that studies  equations with discontinuous diffusion coefficient, written by Garz\'on et al.  \cite{GLT}. There the authors proved the existence and uniqueness of solutions to the SDE driven by the fractional Brownian motion $B^H$ with $H>\frac12$ given by
\begin{equation}\label{DISC}
X_t = X_0 + \int_0^t \sigma (X_s) d B^H_s \quad , \quad t \geq 0,
\end{equation}
where the function $\sigma$ is given by
\begin{equation}\label{sigma1}
\sigma (x) = \frac{1}{\alpha} \1_{x \geq 0} + \frac{1}{1-\alpha} \1_{x < 0}, \ \alpha \in \left(0,\frac12\right).
\end{equation}
The authors showed that the explicit solution to  the equation (\ref{DISC}) is
\begin{equation}\label{SDISC}
X_t =  \alpha B_t^H  \1_{B^H_t > 0} + (1- \alpha ) B^H_t\1_{B^H_t < 0} ,  \quad  \quad t \geq 0.
\end{equation}
It is straightforward to see that the explicit existence and uniqueness of solution to equation (\ref{DISC})  holds also if $\alpha$ and $1-\alpha$ are replaced with $\alpha_+$ and $\alpha_-$ satisfying $ 0 < \alpha_- < \alpha_+$ (or $0 < \alpha_+ < \alpha_-$, respectively). 



One of the reasons why SDEs with discontinuous diffusion coefficient are interesting is their relation to the Skew Brownian motion. In the Brownian motion framework, the Skew Brownian motion appeared  as a natural generalization of the Brownian motion. The Skew Brownian motion  is a process that behaves like a Brownian motion except that the sign of each excursion is chosen using an independent Bernoulli random variable with the parameter $\alpha \in (0 , 1)$.  For $\alpha = 1/2,$ the process corresponds to a Brownian motion. This process is a Markov process and a semi-martingale. Moreover, it is a strong solution to certain SDE with local time (see \cite{Lejay} for a survey). Let
\begin{equation}\label{SB}
X_t = x + B_t + (2\alpha - 1) L^0_t(X),
\end{equation}
where $L_t^0(X)$ is the symmetric local time of $X$ at $0$. In the case of the Brownian motion, it follows  from the  It\^o-Tanaka formula that the equations (\ref{SB}) and (\ref{DISC}) with $\sigma (x) = \frac{1}{\alpha} \1_{\{x \geq 0\}} + \frac{1}{1-\alpha} \1_{ \{x < 0\}}$ are equivalent. For a comprehensive survey on Skew Brownian motion, see the work by Lejay A. in \cite{Lejay}. 

In the case of the fractional Brownian motion, the Tanaka type formulas are more complicated and no relations between the two types of equations are known to exist. The motivation for the authors in \cite{GLT} to study equation \eqref{DISC} stemmed from this fact.

To the best of our knowledge, \cite{LP} is the only study that considers the inference of parameters related to SDE with a discontinuous diffusion process. The study considers the case of a discontinuous
diffusion coefficient that can only attain two different values. More precisely, the authors of \cite{LP} studied the so-called \emph{oscillating Brownian motion} that is a solution to the SDE
\begin{equation}\label{le-pi}
X_t = x + \int_0^t \sigma(X_s) dW_s,
\end{equation}
where $W$ is a standard Brownian motion and $\sigma(x) = \alpha_+\1_{x \ge 0}  + \alpha_- \1_{x < 0}, \quad x \in \mathbb{R}$. 
The authors proposed two natural consistent estimators, which are variants of the integrated volatility
estimator. Moreover, the stable convergence towards certain Gaussian mixture of the renormalised estimators was proven. The estimators are given by 
\begin{eqnarray}\label{LP}
\hat{\alpha}_+ = \sqrt{\frac{\sum_{k=1}^n  \left(X_k - X_{k-1} \right)^2    }{\sum_{k=1}^{n} \1_{X_k \ge0}    }}, \quad 
\hat{\alpha}_- = \sqrt{\frac{\sum_{k=1}^n  \left(X_k - X_{k-1} \right)^2  }{\sum_{k=1}^{n} \1_{X_k \le0}    }}.
\end{eqnarray}
Note that when the paths are
strictly positive or strictly negative, only one of the estimators can be computed.


Motivated by Equation (\ref{SDISC}), we define the Oscillating Gaussian process by 
\begin{equation}
\label{OfBm-b}
X_t = \alpha_+ Y_t\1_{Y_t > 0} + \alpha_- Y_t\1_{Y_t < 0}, t \in T,
\end{equation}
where  $\alpha_+$ and $\alpha_-$ are both strictly positive (or negative, respectively) constants. In addition to the above mentioned links to SDEs and skew Brownian motion, we note that \eqref{OfBm-b} could be applied in various other modelling scenarios as well, making oscillating Gaussian process an interesting object of study. For example, \eqref{OfBm-b} can be viewed as a model for different situations where the variance changes by regions. One of the main interests in this paper is in the estimation of the model parameters $\alpha_+$ and $\alpha_-$. In order to be able to compute estimators for both parameters in all possible cases, we define estimators based on moments and study their asymptotic properties. Moreover, we show that our moment based approach can be applied under a large class of driving Gaussian processes $Y$ in \eqref{OfBm-b}.
 
The rest of the paper is organised as follows. In Section \ref{sec:oscillating}, we introduce the oscillating Gaussian processes and study their basic properties such as moments, covariance structures, and continuity properties. Section \ref{sec:calibration} is devoted to model calibration. We begin by showing that the moment estimators are consistent and satisfy central limit theorems under suitable assumptions on the driving Gaussian process. On top of that, we also consider corresponding estimators based on discrete observations. In Subsection \ref{subsec:ss-oscillating}, we briefly discuss how Lamperti transform can be used to study oscillating Gaussian processes driven by self-similar Gaussian noise, and as a particular example, we apply the method to the case of the bifractional Brownian motion. 
We end the paper with a short summary and a discussion about future prospects.



\section{Oscillating Gaussian processes}
\label{sec:oscillating}
Throughout this section we consider Gaussian oscillating processes $X=(X_t)_{t\geq 0}$ defined by
\begin{equation}
\label{OfBm-general}
X_t = \alpha_+ Y_t\1_{Y_t > 0} + \alpha_- Y_t\1_{Y_t < 0},
\end{equation}
where $Y = (Y_t)_{t\geq 0}$ is a stationary Gaussian process and the $\alpha_+$ and $\alpha_-$ are positive parameters such that $\alpha_+\neq \alpha_-$. Note that the $\alpha_+$ and $\alpha_-$ describe the magnitude of variations of $X$ on different regions. Our goal is to estimate the unknown parameters $\alpha_+$ and $\alpha_-$. In order to do this, we assume that $\E (Y_t^2)=1$. Note that the general case $\E (Y_t^2) = \sigma^2$ can be written as
$$
X_t = \alpha_+ \sigma \tilde{Y}_t\1_{\tilde{Y}_t > 0} + \alpha_-\sigma \tilde{Y}_t\1_{\tilde{Y}_t < 0},
$$
where now $\E (\tilde{Y}_t) = 1$. We also assume that the parameters $\alpha_+$ and $\alpha_-$ are both strictly positive (or negative). 

\begin{remark}
Note that we can extend our analysis in a straightforward manner to the case $\alpha_-<0<\alpha_+$ (or $\alpha_->0>\alpha_+$) as well. Reason for that is that we defined $X$ with \eqref{OfBm-general} directly instead of restricting ourselves to the situation where $X$ is a solution to SDE (\ref{DISC}), in which case the solution is known to exists and is of the form \eqref{OfBm-general} only for $\alpha_-,\alpha_+>0$. See also Remark \ref{rem:extension}. 
\end{remark}

\begin{definition}[Oscillating Gaussian process (OGP)]
Let $Y$ be a centered stationary Gaussian process with variance $\sigma^2=1$ and covariance function $r(t)$, and let $\alpha_+,\alpha_->0, \alpha_+\neq \alpha_-$ be constants. We define the oscillating version $X$ of $Y$ by
\begin{equation}
\label{OfBm}
X_t = \alpha_+ Y_t\1_{Y_t > 0} + \alpha_- Y_t\1_{Y_t < 0}.
\end{equation}
\end{definition}
In the following lemmas we compute the moments and covariances of  the OGP $X$ defined in (\ref{OfBm}).
\begin{lemma}
\label{lemma:moments}
Let $n\geq 1$ be an integer and $t\geq 0$ arbitrary. Then
$$
\mu_n := \E (X_0^n) = \E (X_t^n) = \frac{2^{\frac{n}{2}}\Gamma\left(\frac{n+1}{2}\right)}{2\sqrt{\pi}}(\alpha_+^n+(-1)^n\alpha_-^n).
$$
\end{lemma}
\begin{proof}
By the definition of OGP, we have
$$
X_t^n = \alpha_+^nY_t^n\1_{Y_t > 0} + \alpha_-^n Y_t^n\1_{Y_t < 0}.
$$
Since $Y$ is a centered stationary Gaussian process we have
\begin{equation}
\label{eq:positive-moment}
\E (Y_t^n\1_{Y_t >0}) = \int_0^\infty \frac{x^n}{\sqrt{2\pi}}e^{-\frac{x^2}{2}}dx  =  \frac{1}{2}\int_{-\infty}^\infty \frac{|x|^n}{\sqrt{2\pi}}e^{-\frac{x^2}{2}}dx = \frac{1}{2}\E|N|^n,
\end{equation}
where $N\sim \mathcal{N}(0,1)$. Similarly,
\begin{equation}
\label{eq:negative-moment}
\E (Y_t^n\1_{Y_t <0}) = (-1)^n\frac{1}{2}\E|N|^n.
\end{equation}
Now, the well known formula for a standard normal variable  $\E|N|^n = \frac{2^{\frac{n}{2}}\Gamma\left(\frac{n+1}{2}\right)}{\sqrt{\pi}}$, implies the claim.
\end{proof}
The following lemma allows us to compute the parameters $\alpha_+$ and $\alpha_-$ in terms of the moments.
\begin{lemma}
\label{lemma:parameters}
Let $t>0$ be arbitrary. Then
$$
\alpha_+ = \sqrt{\frac{\pi}{2}}\mu_1 + \frac{1}{2}\sqrt{4\mu_2-2\pi(\mu_1)^2}
$$
and
$$
\alpha_- = -\sqrt{\frac{\pi}{2}}\mu_1 + \frac{1}{2}\sqrt{4\mu_2-2\pi(\mu_1)^2}.
$$
\end{lemma}
\begin{proof}
Since $\Gamma(1) =1$ and $\Gamma\left(\frac{3}{2}\right) = \frac{\sqrt{\pi}}{2}$, Lemma \ref{lemma:moments} yields
$$
\mu_1 = \frac{1}{\sqrt{2\pi}}\left(\alpha_+-\alpha_-\right)
$$
and
$$
\mu_2 = \frac{1}{2}\left(\alpha_+^2+\alpha_-^2\right).
$$
From the first equality we get
$$
\alpha_+ = \alpha_- + \sqrt{2\pi}\mu_1.
$$
Plugging into the second inequality with some simple manipulations gives
\begin{equation}
\label{eq:quadratic}
2\alpha_-^2 + 2\sqrt{2\pi}\mu_1 \alpha_- + 2\pi \mu_1^2 -2\mu_2 = 0.
\end{equation}
Now
$$
4\mu_2 -2\pi\mu_1^2 = 2\alpha_+^2+2\alpha_-^2 - (\alpha_+ - \alpha_-)^2 = (\alpha_+ + \alpha_-)^2 > 0,
$$
and since $\alpha_->0$, we obtain the result.
\end{proof}
\begin{remark}
\label{rem:extension}
Note that in the proof of Lemma \ref{lemma:parameters} we applied the assumption $\alpha_->0$. In the case $\alpha_-<0<\alpha_+$, one has to choose the other solution to Equation \eqref{eq:quadratic} yielding
$$
\alpha_- = -\sqrt{\frac{\pi}{2}}\mu_1 - \frac{1}{2}\sqrt{4\mu_2-2\pi(\mu_1)^2}.
$$
\end{remark}
In the next lemma, we derive the covariance function of the process $X$. That allows us to  obtain consistency for our estimators.
\begin{lemma}
\label{lemma:pos-covariance}
Let $N_1\sim\mathcal{N}(0,1)$ and $N_2\sim\mathcal{N}(0,1)$ such that $Cov(N_1,N_2) = a$.
Then
$$
\E(N_1^mN_2^n\1_{N_1,N_2>0})=2^{\frac{n+m-4}{2}}\pi^{-1}(1-a^2)^{\frac{n+m-1}{2}}\sum_{r=0}^\infty \frac{(4a)^r}{r!}\Gamma\left(\frac{n+r+1}{2}\right)\Gamma\left(\frac{m+r+1}{2}\right)
$$
\end{lemma}
\begin{proof}
We have
$$
\E(N_1^mN_2^n\1_{N_1,N_2>0}) = \frac{1}{2\pi\sqrt{1-a^2}}\int_0^\infty\int_0^\infty x^my^n
e^{-\frac{x^2+y^2-2axy}{2(1-a^2)}}dx dy.
$$
Change of variables $u = \frac{x}{\sqrt{2(1-a^2)}}$ and $v=\frac{y}{\sqrt{2(1-a^2)}}$ gives
$$
\E(N_1^mN_2^n\1_{N_1,N_2>0}) = 2^{\frac{n+m}{2}}\pi^{-1}(1-a^2)^{\frac{n+m-1}{2}}\int_0^\infty\int_0^\infty u^mv^n e^{-u^2-v^2+2auv}du dv,
$$
and using formula  3.5-5 in \cite{rice} we obtain 
$$
\int_0^\infty\int_0^\infty u^mv^n e^{-u^2-v^2+2auv}du dv = \frac{1}{4}\sum_{r=0}^\infty \frac{(4a)^r}{r!}\Gamma\left(\frac{n+r+1}{2}\right)\Gamma\left(\frac{m+r+1}{2}\right).
$$
This proves the claim.
\end{proof}
In the sequel we apply standard Landau notation $O(\cdot)$.
\begin{corollary}
\label{corollary:pos-covariance-asymp}
Let $N_1\sim\mathcal{N}(0,1)$ and $N_2\sim\mathcal{N}(0,1)$ such that $Cov(N_1,N_2) = a$, and let $n\geq 1$ be an integer.
Then
$$
\E(N_1^nN_2^n\1_{N_1,N_2>0})=2^{n-2}\pi^{-1}\Gamma\left(\frac{n+1}{2}\right)^2 + O(|a|)
$$
and
$$
\E(N_1^nN_2^n\1_{N_1>0,N_2<0})=(-1)^n 2^{n-2}\pi^{-1}\Gamma\left(\frac{n+1}{2}\right)^2 + O(|a|).
$$
\end{corollary}
\begin{proof}
It follows from Lemma \ref{lemma:pos-covariance} that
$$
\E(N_1^nN_2^n\1_{N_1,N_2>0}) = 2^{n-2}\pi^{-1}\Gamma\left(\frac{n+1}{2}\right)^2 (1-a^2)^{\frac{2n-1}{2}}+ O(|a|).
$$
Now, the first claim follows from the fact that
$$
(1-a^2)^{\frac{2n-1}{2}} = 1 + O(|a|).
$$
The second claim follows similarly since
$$
\E(N_1^nN_2^n\1_{N_1>0,N_2<0}) = (-1)^n\E(N_1^n(-N_2)^n\1_{N_1>0,-N_2>0}).
$$
\end{proof}
\begin{corollary}
\label{corollary:covariance-X}
Let $X$ be the oscillating Gaussian process defined in \eqref{OfBm}. Then
$$
Cov(X_t^n,X_s^n) =O(|r(t-s)|),
$$
where $r$ is the covariance function of $Y$.
\end{corollary}
\begin{proof}
We have
\begin{equation*}
\begin{split}
X_t^n X_s^n
&= \alpha_+^{2n}Y_t^nY_s^n\1_{Y_t,Y_s > 0} + \alpha_-^{2n} Y_t^nY_s^n\1_{Y_t,Y_s < 0} \\
&+
\alpha^n_+\alpha^n_-(Y_t^nY_s^n\1_{Y_t>0,Y_s<0}+Y_t^nY_s^n\1_{Y_t<0,{Y}_s>0}).
\end{split}
\end{equation*}
Taking expectation and using Corollary \ref{corollary:pos-covariance-asymp} we get
$$
\E(X_t^n X_s^n) = 2^{n-2}\pi^{-1}\Gamma\left(\frac{n+1}{2}\right)^2\left(\alpha_+^{2n}+\alpha_-^{2n} + 2(-1)^n\alpha^n_+\alpha^n_-\right) + O(|r(t-s)|).
$$
Lemma \ref{lemma:moments} now implies the claim.
\end{proof}
We end this section with the following result that ensures the path continuity of the OGP $X$.
\begin{proposition}
\label{prop:holder}
Let $X$ be the oscillating Gaussian process defined by \eqref{OfBm}. If $Y$ has H\"older continuous paths of order $\gamma\in(0,1]$ almost surely, then so does $X$.
\end{proposition}
\begin{proof}
The result follows from the simple observations that
\begin{equation*}
\begin{split}
&|Y_t\1_{Y_t>0} - Y_s\1_{Y_s>0}| \\
&=Y_t\1_{Y_t>0\geq Y_s} + Y_s\1_{Y_s>0\geq Y_t} + |Y_t-Y_s|\1_{Y_t,Y_s>0}\\
&\leq (Y_t-Y_s)\1_{Y_t>0\geq Y_s} + (Y_s-Y_t)\1_{Y_s>0\geq Y_t} + |Y_t-Y_s|\1_{Y_t,Y_s>0}\\
&\leq |Y_t-Y_s|\left(\1_{Y_t>0\geq Y_s} + \1_{Y_s>0\geq Y_t} +\1_{Y_t,Y_s>0}\right)\\
&\leq |Y_t-Y_s|.
\end{split}
\end{equation*}
Similarly,
$$
|Y_t\1_{Y_t<0} - Y_s\1_{Y_s<0}| \leq |Y_t-Y_s|
$$
from which the claim follows.
\end{proof}
\section{Model calibration}
\label{sec:calibration}
This section is devoted to the estimation of the unknown parameters $\alpha_+,\alpha_-$ by the method of moments. Following the  ideas of Lemma \ref{lemma:parameters}, we define
\begin{equation}
\label{eq:alpha-plus_estimator}
\hat\alpha_+(T) = \sqrt{\frac{\pi}{2}}\hat\mu_1(T)
 + \frac{1}{2}\sqrt{\left|4\hat\mu_2(T)-2\pi\hat\mu^2_1(T)\right|}
\end{equation}
and
\begin{equation}
\label{eq:alpha-minus_estimator}
\hat\alpha_-(T) = \hat\alpha_+(T) - \sqrt{\frac{\pi}{2}}\hat\mu_1(T),
\end{equation}
where $\hat\mu_i(T), i=1,2$ are the classical moment estimators defined by
\begin{equation}
\label{eq:moment-estimator}
\hat\mu_{i}(T) = \frac{1}{T} \int_0^{T} X_u^i du.
\end{equation}
\begin{remark}
Note that here we have taken absolute values inside the square roots in order to obtain real valued estimates for real valued quantities. Since
$$
4\mu_2 - 2\pi\mu_1^2 > 0,
$$
this does not affect the asymptotical properties of the estimators.
\end{remark}
The following result gives us the consistency and can be viewed as one of our main theorems. The proof is postponed to Subsection \ref{subsec:proofs}.
\begin{theorem}
\label{thm:estimation-consistency}
Assume that $|r(T)|\to 0$ as $T\to \infty$. Then, for any $p\geq 1$, we have
$$
\hat\alpha_+(T) \to \alpha_+
$$
and
$$
\hat\alpha_-(T) \to \alpha_-
$$
in $L^p$, as $T \to \infty$.
\end{theorem}
In order to study the limiting distribution, we need some additional assumptions on the covariance function $r$.
\begin{assumption}
\label{assumption:covariance}
Let $r$ be the covariance function of $Y$. We assume that one of the following condition hold:
\begin{enumerate}
\item The covariance function $r$ satisfies $r\in L^1(\R)$.
\item We have that
$$
\lim_{t\to\infty}\frac{r(t)}{t} = C < \infty.
$$
\item There exists $H\in\left(\frac12,1\right)$ such that
$$
\lim_{t\to\infty}\frac{r(t)}{t^{2H-2}} = C < \infty.
$$
\end{enumerate}
\end{assumption}
\begin{remark}
The first condition in Assumption \ref{assumption:covariance} corresponds to short-range dependence and the last condition corresponds to long-range dependence. The second condition corresponds to the border case resulting to a logarithmic factor to our normalising sequence (see Theorem \ref{thm:estimation-clt}).
\end{remark}
 The following theorem gives the central limit theorem for the moments estimators.
\begin{theorem}
\label{thm:estimation-moment-clt}
Let $\hat\mu_1(T)$ and $\hat\mu_2(T)$ be defined by \eqref{eq:moment-estimator}, and let $\hat\mu(T) = (\hat\mu_1(T),\hat\mu_2(T))$ and $\mu = (\mu_1,\mu_2)$. Then,
\begin{enumerate}
\item \label{c1}  
if $r$ satisfies the condition (1) of Assumption \ref{assumption:covariance},
$$
\sqrt{T}\left(\hat\mu(T) - \mu\right) \to \mathcal{N}(0,\Sigma_1^2)
$$
in law as $T \to \infty$,
\item \label{c2}  
if $r$ satisfies the condition (2) of Assumption \ref{assumption:covariance},
$$
\sqrt{\frac{T}{\log T}}\left(\hat\mu(T)-\mu\right) \to \mathcal{N}(0,\Sigma_2^2)
$$
in law as $T \to \infty$, and
\item \label{c3}  
if $r$ satisfies the condition (3) of Assumption \eqref{assumption:covariance},
$$
T^{1-H}\left(\hat\mu(T)-\mu\right) \to \mathcal{N}(0,\Sigma_3^2)
$$
in law as $T \to \infty$,
\end{enumerate}
where $\Sigma_1^2$, $\Sigma_2^2$, and $\Sigma_3^2$ are constant covariance matrices depending on $\alpha_+,\alpha_-$, and the covariance $r$.
\end{theorem}
\begin{remark}
Note that the covariance matrices $\Sigma^2_i,i=1,2,3$ in Theorem \ref{thm:estimation-moment-clt} can be calculated explicitly in terms of the covariance $r$, $\alpha_+$, and $\alpha_-$ by computing the chaos decompositions of the functions
$
f_1(x) = \alpha_+ x\1_{x>0} + \alpha_- x\1_{x<0}
$
and $f_2(x) =\alpha_+ x^2\1_{x>0} + \alpha_- x^2\1_{x<0}$.
\end{remark}
\begin{remark}
By replacing $\hat\mu_n(T)$ with
$$
\hat\mu_n(t,T) = \frac{1}{T}\int_0^{tT} X_u^n du
$$
and normalising accordingly, one can obtain functional versions of the above limit theorems. That is, in cases (\ref{c1}) and (\ref{c2}) of Theorem \ref{thm:estimation-moment-clt},  we obtain convergence in law in the space of continuous functions towards $\sigma W_t$, where $W_t$ is a Brownian motion. In the case (\ref{c3}), the limiting process is $\sigma B^H_t$, where $B^H$ is the fractional Brownian motion. Indeed, the last case follows from a classical result by Taqqu \cite{taqqu} and the first case from \cite{nourdin-nualart} and from the fact that all moments of $X$ are finite. However, from practical point of view, translating these results to functional versions of the estimators $\hat\alpha_+(T)$ and $\hat\alpha_-(T)$ is not feasible.
Indeed, this follows from the fact that in the functional central limit theorem for $\hat\mu(t,T)$ the normalisation (subtracting the true value) is done inside the integral, while for $\hat\alpha_+(T)$ and
$\hat\alpha_-(T)$ this is done after integration.
\end{remark}
Theorems \ref{thm:estimation-consistency} and \ref{thm:estimation-moment-clt} now give us the following limiting distributions for the estimators $\alpha_+(T)$ and $\alpha_-(T)$.
\begin{theorem}
\label{thm:estimation-clt}
Let $\hat\alpha_+(T)$ and $\hat\alpha_-(T)$ be defined by \eqref{eq:alpha-plus_estimator} and \eqref{eq:alpha-minus_estimator}, respectively, and let $\hat\alpha(T) = (\hat\alpha_+(T),\hat\alpha_-(T))$ and $\alpha = (\alpha_+,\alpha_-)$. Then,
\begin{enumerate}
\item
if $r$ satisfies the condition (1) of Assumption \ref{assumption:covariance},
$$
\sqrt{T}\left(\hat\alpha(T)-\alpha\right) \to \mathcal{N}(0,\Sigma_A^2)
$$
in law,
\item
if $r$ satisfies the condition (2) of Assumption \ref{assumption:covariance}, then
$$
\sqrt{\frac{T}{\log T}}\left(\hat\alpha(T)-\alpha\right) \to \mathcal{N}(0,\Sigma_B^2)
$$
in law, and
\item
if $r$ satisfies the condition (3) of Assumption \ref{assumption:covariance}, then
$$
T^{1-H}\left(\hat\alpha(T)-\alpha\right) \to \mathcal{N}(0,\Sigma_C^2)
$$
in law,
\end{enumerate}
where $\Sigma_A^2$, $\Sigma_B^2$, and $\Sigma_C^2$ are constant covariance matrices depending on $\alpha_+,\alpha_-$, and the covariance $r$.
\end{theorem}
\begin{proof}
The result follows from Theorems \ref{thm:estimation-consistency} and \ref{thm:estimation-moment-clt}
together with a simple application of a multidimensional delta method. We leave the details to the reader.
\end{proof}
\begin{remark}
As in the case of Theorem \ref{thm:estimation-moment-clt}, the covariance matrices $\Sigma^2_j,j=A,B,C$ in Theorem \ref{thm:estimation-clt} can be calculated explicitly. Indeed, by utilising two-dimensional delta method, $\Sigma^2_j,j=A,B,C$ are linear transformations of $\Sigma^2_i,i=1,2,3$ defined in Theorem \ref{thm:estimation-moment-clt}. 
\end{remark}

\subsection{Proofs of Theorems \ref{thm:estimation-consistency} and \ref{thm:estimation-moment-clt}}
\label{subsec:proofs}
We begin with the following versions of weak law of large numbers.

\begin{proposition}[Laws of large numbers]
\label{prop:wlln}
Let $n\geq 1$ and suppose that $|r(T)| \to 0$ as $|T| \to \infty$. Then, for any $p\geq 1$, as $T\to \infty$,
$$
\frac{1}{T}\int_0^{T} X_u^n du \to \frac{2^{\frac{n}{2}}\Gamma\left(\frac{n+1}{2}\right)}{2\sqrt{\pi}}(\alpha_+^n+(-1)^n\alpha_-^n)
$$
in $L^p$ as $T \to \infty$.
\end{proposition}
\begin{proof}
In order to prove the claim, we have to show that
$$
\left\Vert\frac{1}{T}\int_0^T X_u^n - \mu_n du\right\Vert_p \to 0,
$$
where $\Vert \cdot\Vert_p$ is the $p$:th norm.
We first observe that it suffices to prove convergence in probability. Indeed, for every $p\geq 1$ and $\epsilon>0$, we have
$$
\sup_{T\geq 1}\left\Vert\frac{1}{T}\int_0^T X_u^n - \mu_n du\right\Vert_{p+\epsilon} \leq \sup_{T\geq 1}\frac{1}{T}\int_0^T \left\Vert X_u^n - \mu_n\right\Vert_{p+\epsilon} du \leq C.
$$
Thus, for every $p$, the quantity
$$
\left|\frac{1}{T}\int_0^T X_u^n - \mu_n du\right|^{p}
$$
is uniformly integrable. Now the result follows from the fact that uniform integrability and convergence in probability implies convergence in $L^1$, i.e.
$$
\E\left|\frac{1}{T}\int_0^T X_u^n - \mu_n du\right|^{p} \to 0, \quad \mbox{as}\quad T \to \infty.
$$
Let us now prove the convergence in $L^2$, which then implies the convergence in probability.
By Corollary \ref{corollary:covariance-X}, we have that
$$
\E\left|\frac{1}{T}\int_0^T X_u^n du - \frac{2^{\frac{n}{2}}\Gamma\left(\frac{n+1}{2}\right)}{2\sqrt{\pi}}(\alpha_+^n+(-1)^n\alpha_-^n)\right|^2 = T^{-2}\int_0^T\int_0^T a(u,s)du ds,
$$
where $a(u,s) = O(|r(s-u)|)$.
Writing
\begin{equation*}
\begin{split}
&\int_{(u,s)\in[0,T]^2}r(u-s)du ds \\
&= \int_{(u,s)\in[0,T]^2,|u-s|\geq T_0}r(u-s)du ds + \int_{(u,s)\in[0,T]^2,|u-s|<T_0}r(u-s)du ds
\end{split}
\end{equation*}
and choosing $T_0$ such that $|r(u-s)|<\epsilon$ on $\{(u,s)\in[0,T]^2,|u-s|\geq T_0\}$ yields the result.
\end{proof}
\begin{proof}[Proof of Theorem \ref{thm:estimation-consistency}]
By Proposition \ref{prop:wlln}, we have that
$(\hat\mu_1(T),\hat\mu_2(T)) \to (\mu_1,\mu_2)$ in $L^p$ as $T\to \infty$. As $ \displaystyle \sup_{T\geq 1}\Vert \hat\mu_1(T)\Vert_p < \infty$ for all $p\geq 1$, it follows from H\"older inequality that, for any $r>0$, we have
$$
\Vert \hat\mu^2_1(T) - \mu_1^2\Vert_p = \Vert (\hat\mu_1(T) + \mu_1)(\hat\mu_1(T)-\mu_1)\Vert_p \leq C \Vert \hat\mu_1(T)-\mu_1\Vert_{p+r},
$$
where $C$ is a  constant.
Thus
$$
\Vert \hat\mu^2_1(T) - \mu_1^2\Vert_p \to 0 \quad \mbox{as} \quad T \to \infty.
$$
Now, using $|\sqrt{a}-\sqrt{b}| \leq \sqrt{|a-b|}$ and the triangle inequality, we get
\begin{equation*}
\begin{split}
&\sqrt{\left|4\hat\mu_2(T)-2\pi\hat\mu^2_1(T)\right|} - \sqrt{\left|4\mu_2-2\pi\mu^2_1\right|} \\
&\leq C\sqrt{|\hat\mu_2(T)-\mu_2|} + C\sqrt{|\hat\mu_1^2(T) - \mu_1^2|}.
\end{split}
\end{equation*}
The claim now follows from the fact that, for any random variable $Z$ and for any $p\geq 2$,
$$
\Vert \sqrt{|Z|}\Vert_p = \sqrt{\Vert Z\Vert_{p/2}}.
$$
\end{proof}
We proceed now to the proof of Theorem \ref{thm:estimation-clt}. Before that we recall some preliminaries.

Let $N\sim \mathcal{N}(0,1)$ and let $f$ be a function such that $\E \left( f(N)^2\right) < \infty$. Then $f$ admits the Hermite decomposition
$$
f(x) = \sum_{k=0}^\infty \beta_k H_k(x),
$$
where $H_k,k=0,1,\ldots$ are the Hermite polynomials. The index $d=\min\{k\geq 1: \beta_k \neq 0\}$ is called the Hermite rank of $f$. For our purposes we need to consider the functions
$$
f_i(x) = \alpha^i_+ x^i \1_{x>0} + \alpha_- x^i \1_{x<0}, \quad i=1,2.
$$
The Hermite decompositions of $f_1$ and $f_2$ are denoted by
\begin{equation}
\label{eq:f1}
f_1(x) = \sum_{k=0} \beta_{1,k}H_k(x)
\end{equation}
and
\begin{equation}
\label{eq:f2}
f_2(x) = \sum_{k=0}\beta_{2,k}H_k(x),
\end{equation}
respectively.
\begin{proof}[Proof of Theorem \ref{thm:estimation-moment-clt}]
By Cramer-Wold device, it suffices to prove that
each linear combination
$$
Z(y_1,y_2,T):= y_1(\hat\mu_1(T) - \mu_1) + y_2(\hat\mu_2(T)-\mu_2),
$$
when properly normalised, converges towards a normal distribution. By using representations \eqref{eq:f1} and \eqref{eq:f2}, it follows that $Z(y_1,y_2,T)$ have representation
\begin{equation}
\label{eq:Z-rep}
Z(y_1,y_2,T) = \frac{1}{T}\int_0^T \sum_{k=0}^\infty \gamma_k H_k(Y_t)dt,
\end{equation}
where
$
\gamma_k = y_1\beta_{1,k} + y_2\beta_{2,k}.
$
Note also that we have
$
\E\hat\mu_i(T)  = \mu_i,\quad i=1,2,
$
and thus $\gamma_0 = 0$, i.e. $Z(y_1,y_2,T)$ is a normalised sequence.
We begin with the first case that is relatively easy. Indeed,
suppose that the condition (1) of Assumption \ref{assumption:covariance} holds. Then, as $r$ is integrable, continuous version of the Breuer-Major theorem (see e.g. \cite{nourdin-nualart}) implies the claim directly.

Under the other two conditions, we first note that the only contributing factor to the limiting distribution in \eqref{eq:Z-rep} is
$$
\frac{1}{T}\int_0^T \gamma_1H_1(Y_t)dt.
$$
This follows from the fact that
\begin{equation*}
\begin{split}
\E \left[\sum_{k=2}^\infty \gamma_k \int_0^T H_k(Y_t)dt\right]^2 \leq CT\int_0^T r^2(u)du
\end{split}
\end{equation*}
and clearly
$$
\frac{1}{\log T}\int_0^T r^2(u)du \to 0
$$
under the condition (2) and
$$
T^{1-2H}\int_0^T r^2(u)du \to 0
$$
under the condition (3). Thus it suffices to prove that
$$
[y_1\beta_{1,1}+y_2\beta_{2,1}]\frac{l(T)}{T}\int_0^T Y_t dt
$$
converges towards normal distribution, where $l(T) = \sqrt{\frac{T}{\log T}}$ under the condition (2) and $l(T) = T^{1-H}$ under the condition (3). Convergence of $\frac{l(T)}{T}\int_0^T Y_t dt$ follows from the fact that $Y$ is Gaussian, and the variance converges. Indeed, we have that
\begin{equation*}
\begin{split}
\E \left(\int_0^T Y_t dt\right)^2 &= \int_0^T \int_0^T \E (Y_uY_s)du ds \\
&= \int_0^T \int_0^T r(u-s)du ds \\
&= \int_0^T r(u)(T-u)du.
\end{split}
\end{equation*}
Under conditions (2) and (3) of Assumption \ref{assumption:covariance} we obtain that, in both cases,
$$
\frac{l^2(T)}{T^2}\int_0^T r(u)(T-u)du \to C>0.
$$
Thus, it suffices to prove that $\beta_{1,1}\neq 0$ or $\beta_{1,2}\neq 0$. Recall that
$$
f_1(x) = \alpha_+ x\1_{x>0}+\alpha_- x\1_{x<0}.
$$
Thus we have
$$
\beta_{1,1} = \E [f_1(N)N],
$$
where $N\sim\mathcal{N}(0,1)$. Using \eqref{eq:positive-moment} and \eqref{eq:negative-moment} we get
$$
\beta_{1,1} = \frac{3}{2}\left(\alpha_+-\alpha_-\right).
$$
Recalling that $\alpha_+\neq \alpha_-$ concludes the proof.
\end{proof}
\begin{remark}
Note that the proof of Theorem \ref{thm:estimation-clt} relied on the fact that $\alpha_+ \neq \alpha_-$. If $\alpha_+ = \alpha_- = \alpha$, then
$
X_t = \alpha|Y_t|
$
and it follows that $\gamma_1=0$ and $\gamma_2\neq 0$. Then, under conditions (1) and (2), the limiting distribution is normal and the rate is $\sqrt{T}$. Under the condition (3), the limiting distribution and the rate depends on the value of $H$. If $H<\frac34$, the limiting distribution is normal and the rate is $\sqrt{T}$. If $H=\frac34$, then the limiting distribution is still normal, but the rate is $\sqrt{\frac{T}{\log T}}$. For $H>\frac34$, the limiting distribution is the Rosenblatt distribution (multiplied by a constant) and the rate is $T^{2-2H}$.
\end{remark}
\subsection{Estimation based on discrete observations}
In practice, one does not observe the continuous path of $X$. Instead of that, one observes $X$ on some discrete time points $0\leq t_0 <t_1 < \ldots <T_N<\infty$. That is why, in practical applications, the integrals in \eqref{eq:moment-estimator} are approximated by discrete sums. Thus the natural moment estimators $\tilde{\mu}_n(N)$ are defined by
\begin{equation}
\label{eq:moment-estimator-discrete}
\tilde{\mu}_n(N) = \frac{1}{T_N}\sum_{k=1}^N X^n_{t_{k-1}}\Delta t_k,
\end{equation}

where $\Delta t_k = t_k - t_{k-1}$. The corresponding estimators $\tilde{\alpha}_+(N)$ and $\tilde{\alpha}_-(N)$ for parameters $\alpha_+$ and $\alpha_-$ are
\begin{equation}
\label{eq:alpha-plus_estimator-discrete}
\tilde\alpha_+(N) = \sqrt{\frac{\pi}{2}}\tilde\mu_1(N)
 + \frac{1}{2}\sqrt{\left|4\tilde\mu_2(N)-2\pi\tilde\mu^2_1(N)\right|}
\end{equation}
and
\begin{equation}
\label{eq:alpha-minus_estimator-discrete}
\tilde\alpha_-(N) = \tilde\alpha_+(N) - \sqrt{\frac{\pi}{2}}\tilde\mu_1(N).
\end{equation}
Let $\Delta_N = \max_k \Delta t_k$. In order to obtain consistency and asymptotic normality for the discretised versions, we have to assume that $T_N \to \infty$ and, at the same time, that $\Delta_N \to 0$ in a suitable way. The following proposition studies the difference between $\hat\mu_n(T_N)$ and $\tilde\mu_n(N)$.
\begin{proposition}
\label{prop:discrete-approx}
Denote the variogram of the stationary process $Y$ by $c(t)$, i.e.
$$
c(t) = 2\left[r(0)-r(t)\right],
$$
where $r$ is the covariance function.
Then, for any $n\geq 1$ and for any $p\geq 1$, there exists a constant $C=C(n,p,\alpha_+,\alpha_-)$ such that
$$
\left\Vert\hat\mu_n(T_N)-\tilde\mu_n(N)\right\Vert_p \leq C\sup_{0\leq t \leq \Delta_N}\sqrt{c(t)}.
$$
\end{proposition}
\begin{proof}

We have, by Minkowski inequality, that \begin{equation*}
\begin{split}
&\left\Vert \hat\mu_n(T_N)-\tilde\mu_n(N)\right\Vert_p\\
&= \left\Vert \frac{1}{T_N}\int_0^{T_N}X_u^n du - \frac{1}{T_N}\sum_{k=1}^N X^n_{t_{k-1}}\Delta t_k\right\Vert_p \\
&\leq  \frac{1}{T_N} \sum_{k=1}^N \int_{t_{k-1}}^{t_k}\left\Vert X_u^n - X_{t_{k-1}}^n\right\Vert_p du.
\end{split}
\end{equation*}
Using,
$$
x^n-y^n = (x-y)\sum_{j=0}^{n-1}x^jy^{n-1-j},
$$
we get, for any $s,u\geq 0$, that
$$
|X_s^n-X_u^n| \leq |X_s-X_u|\sum_{j=0}^{n-1}|X_s|^j|X_u|^{n-1-j}.
$$
Thus, a repeated application of H\"older inequality together with the fact that $\sup_{s\geq 0}\Vert X_s\Vert_p < \infty$ implies that, for every $q>p$, we have
$$
\left\Vert X_u^n - X_{s}^n\right\Vert_p \leq C\Vert X_u - X_s\Vert_{q},
$$
where $C$ is a constant.
Moreover, by the proof of Proposition \ref{prop:holder}, we have,
$$
|X_u-X_s| \leq C|Y_u-Y_s|.
$$
Since $Y$ is Gaussian, hypercontractivity implies that
$$
\Vert X_u - X_s\Vert_q \leq C\Vert Y_u - Y_s\Vert_2.
$$
Now stationarity of $Y$ gives
$$
\Vert Y_u -Y_s\Vert_2 = \sqrt{c(u-s)}.
$$
Thus we observe
\begin{equation*}
\begin{split}
&\frac{1}{T_N} \sum_{k=1}^N \int_{t_{k-1}}^{t_k}\left\Vert X_u^n - X_{t_{k-1}}^n\right\Vert_p du \\
& \leq \frac{C}{T_N} \sum_{k=1}^N \int_{t_{k-1}}^{t_k}\sqrt{c(u-t_{k-1})} du\\
&\leq C \sup_{0\leq t \leq \Delta_N}\sqrt{c(t)}
\end{split}
\end{equation*}
proving the claim.
\end{proof}
We can now easily deduce the following results for the asymptotical properties of the estimators $\tilde\alpha_+$ and $\tilde\alpha_-$.
\begin{theorem}
\label{thm:estimation-consistency-discrete}
Let $\tilde\alpha_+(N)$ and $\tilde\alpha_-(N)$ be defined by \eqref{eq:alpha-plus_estimator-discrete} and \eqref{eq:alpha-minus_estimator-discrete}, respectively. Suppose that $r(T) \to 0$ as $T\to \infty$ and that $\sup_{0\leq s\leq T}c(s)\to 0$ as $T\to 0$. If $T_N \to \infty$ and $\Delta_N \to 0$ as $N \to \infty$, then for any $p\geq 1$,
$$
\tilde\alpha_+(N) \to \alpha_+
$$
and
$$
\tilde\alpha_-(N) \to \alpha_-
$$
in $L^p$.
\end{theorem}
\begin{proof}
Using the arguments of the proof of Theorem \ref{thm:estimation-consistency} together with Proposition \ref{prop:discrete-approx} we deduce that
$$
\Vert \tilde\alpha_+(N) - \hat\alpha_+(T_N)\Vert_p \to 0
$$
and
$$
\Vert \tilde\alpha_-(N) - \hat\alpha_-(T_N)\Vert_p \to 0.
$$
Thus the claim follows from Theorem \ref{thm:estimation-consistency}.
\end{proof}
\begin{theorem}
\label{thm:estimation-clt-discrete}
Let $\tilde\alpha_+(N)$ and $\tilde\alpha_-(N)$ be defined by \eqref{eq:alpha-plus_estimator-discrete} and \eqref{eq:alpha-minus_estimator-discrete}, respectively, and let $\tilde\alpha(N) = (\tilde\alpha_+(N),\tilde\alpha_-(N))$ and $\alpha = (\alpha_+,\alpha_-)$. Let $\Sigma_A^2$, $\Sigma_B^2$, and $\Sigma_C^2$ be the same covariance matrices as in Theorem \ref{thm:estimation-clt}.
Suppose further that $\displaystyle \sup_{0\leq s\leq t}c(s)\to 0$ as $t\to 0$, $T_N \to \infty$, and $\Delta_N \to 0$ as $N \to \infty$. Denote
$$
h(N) = \sup_{0\leq s \leq \Delta_N}\sqrt{c(s)}.
$$
Then,
\begin{enumerate}
\item \label{a1}
if $r$ satisfies the condition (1) of Assumption \ref{assumption:covariance},
$$
\sqrt{T_N}\left(\tilde\alpha(N)-\alpha\right) \to \mathcal{N}(0,\Sigma_A^2)
$$
in law for every partitions $0\leq t_0 < \ldots T_N$ satisfying $\sqrt{T_N}h(N)\to 0$,
\item \label{a2}
if $r$ satisfies the condition (2) of Assumption \ref{assumption:covariance},
$$
\sqrt{\frac{T_N}{\log T_N}}\left(\tilde\alpha(N)-\alpha\right) \to \mathcal{N}(0,\Sigma_B^2)
$$
in law for every partitions $0\leq t_0 < \ldots T_N$ satisfying $\sqrt{\frac{T_N}{\log T_N}}h(N)\to 0$, and 
\item \label{a3}
if $r$ satisfies the condition (3) of Assumption \ref{assumption:covariance} ,
$$
T_N^{1-H}\left(\tilde\alpha(N)-\alpha\right) \to \mathcal{N}(0,\Sigma_C^2)
$$
in law for every partitions $0\leq t_0 < \ldots T_N$ satisfying $T_N^{1-H}h(N)\to 0$.
\end{enumerate}
\end{theorem}
\begin{proof}
The additional conditions on the mesh together with Proposition \ref{prop:discrete-approx} guarantee that
$$
l(T_N)\Vert \tilde\alpha(N) - \hat\alpha(T_N)\Vert_p \to 0,
$$
where $l(T_N)$ is the corresponding normalisation for each case. Thus the result follows directly from Theorem \ref{thm:estimation-clt}.
\end{proof}
One natural way for choosing the observation points such that the above mentioned conditions are fulfilled, is to choose $N$ equidistant points with $\Delta_N = \frac{\log N}{N}$. Then $\Delta_N \to 0$ and $T_N = N\Delta_N = \log N \to \infty$. If, in addition, $Y$ is H\"older continuous of some order $\theta>0$, then also the rest of the requirements are satisfied. Indeed, it follows from \cite[Theorem 1]{gauss-holder} that if $Y$ is H\"older continuous of order $\theta>0$, then for any $\epsilon>0$, we have
$$
c(t) \leq Ct^{\theta-\epsilon}
$$
for some constant $C$. Thus $h(N)\leq \sqrt{C}\Delta_N^{\frac12(\theta-\epsilon)}$, from which it is easy to see that, for $\epsilon<\theta$,
$$
T_N^{1-H}h(N) \leq \sqrt{\frac{T_N}{\log T_N}}h(N) \leq \sqrt{T_N}h(N) \leq \sqrt{C}\frac{(\log N)^{\frac12(1+\theta-\epsilon)}}{N^{\frac12(\theta-\epsilon)}} \to 0.
$$
\subsection{Oscillating Self-similar Gaussian Processes}
\label{subsec:ss-oscillating}
Self-similar processes form an interesting and applicable class of stochastic processes.
In this subsection, we consider oscillating Gaussian processes driven by self-similar Gaussian processes $Y$. In other words, we consider processes of the type
$$
X_t = \alpha_+ Y_t\1_{Y_t>0} + \alpha_-Y_t\1_{Y_t<0},
$$
where $Y$ is $H$-self-similar for some $H>0$. That is, for every $a>0$, the finite dimensional distributions of the processes $(Y_{at})_{t\geq 0}$ and $(a^HY_t)_{t\geq 0}$ are the same. Throughout this section we assume that we have observed $X_t$ on an interval $[0,1]$, and our aim is to estimate $\alpha_+$ and $\alpha_-$. The key ingredient is the Lamperti transform
\begin{equation}
\label{eq:lamperti}
U_t = e^{-Ht}Y_{e^t}.
\end{equation}
It is well-known that $U$ is stationary on $(-\infty,0]$. Moreover, for $t\geq 0$, we define a process
$$
\tilde{X}_t := e^{Ht}X_{e^{-t}} = \alpha_+ U_{-t} \1_{U_{-t} >0} + \alpha_- U_{-t}\1_{U_{-t}<0}.
$$
Clearly, observing $X$ on $[0,1]$ is equivalent to observing $\tilde{X}_t$ on $t\geq 0$. This leads to the ''moment estimators'' $\hat\mu_i(T)$ defined by
\begin{equation}
\label{eq:moment-estimator-ss}
\hat\mu_i(T) = \frac{1}{T}\int_{e^{-T}}^1 u^{-H-1}X_u^idu.
\end{equation}
The corresponding parameter
estimators $\hat\alpha_+(T)$ and $\hat\alpha_-(T)$ are defined by plugging in $\hat\mu_1(T)$ and $\hat\mu_2(T)$ into \eqref{eq:alpha-plus_estimator} and \eqref{eq:alpha-minus_estimator}, respectively.
Indeed, a change of variable $u=e^{t}$ gives
$$
\hat\mu_i(T)=\frac{1}{T}\int_0^T \tilde{X}^i_t dt.
$$
Thus studying the covariance function $r$ of a stationary Gaussian process $U$ given by \eqref{eq:lamperti} enables us to apply Theorems \ref{thm:estimation-consistency} and \ref{thm:estimation-clt}.

\subsection{The case of bi-fractional Brownian motion}
We end this section with an interesting example. We consider bifractional Brownian motions that, among others, cover fractional Brownian motions and standard Brownian motions. Recall that a bifractional Brownian motion $B^{H,K}$ with $H\in(0,1)$ and $K\in(0,2)$ such that $HK\in(0,1)$ is a centered Gaussian process with a covariance function
$$
R(s,t)=\frac{1}{2^K}\left[(t^{2H}+s^{2H})^K-|t-s|^{2HK}\right].
$$
It is known that $B^{H,K}$ is $HK$-self-similar. Furthermore, one recovers fractional Brownian motion by plugging in $K=1$, from which standard Brownian motion is recovered by further setting $H=\frac12$. Now the covariance function $r$ of the Lamperti transform $U_t=e^{-HKt}B^{H,K}_{e^t}$ has exponential decay (see \cite{langevin}). Thus, we may apply the item (1) of Theorem \ref{thm:estimation-clt} to obtain that
$
\sqrt{T}(\hat\alpha_+(T)-\alpha_+)
$
and
$\sqrt{T}(\hat\alpha_-(T)-\alpha_-)$
are asymptotically normal. Similarly, discretising the integral in \eqref{eq:moment-estimator-ss} and applying Theorems \ref{thm:estimation-consistency-discrete} and \ref{thm:estimation-clt-discrete} allows us to consider parameter estimators based on discrete observations. We leave the details to the reader.

\section{Discussion}
In this paper we considered oscillating Gaussian processes and introduced a moment based estimators for the model parameters. Moreover, we proved consistency and asymptotic normality of the estimators under natural assumptions on the driving Gaussian process. 
An interesting and natural extension to our approach would be to consider oscillating processes with several (more than two) parameters and corresponding regions. This would make the model class more flexible and adaptive. Another topic for future research would be to develop testing procedures for the model parameters.

\subsection*{Acknowledgements}

S. Torres is partially supported by the Project Fondecyt N. 1171335. P. Ilmonen and L. Viitasaari wishes to thank Vilho, Yrj\"o, and Kalle V\"ais\"al\"a foundation for financial support.


\end{document}